\documentclass[11pt,twoside]{amsart}

\usepackage[english]{babel}
\usepackage{amsmath, amsthm, amssymb, amsfonts}
\usepackage{latexsym}
\usepackage{graphicx}
\usepackage{pdfpages}
\usepackage {bm}
\usepackage {indentfirst} 

\usepackage[pagebackref,colorlinks=true,linkcolor=blue,citecolor=blue,urlcolor=red, hypertexnames=true]{hyperref}

\advance\hoffset by -1.5cm

\newcommand{\D}{\mathcal{D}}
\newcommand{\C}{\mathbb{C}}
\newcommand{\N}{\mathbb{N}}

\newcommand{\Z}{\mathbb{Z}}

\newcommand{\B}{\mathcal{B(H)}}

\newcommand{\h}{\mathcal{H}}

\newcommand{\ka}{\mathcal{K}}

\newtheorem{example}{Example}[section]
\newtheorem{theorem}[example]{Theorem}
\newtheorem{lemma}[example]{Lemma}

\newtheorem{proposition}[example]{Proposition}
\newtheorem{corollary}[example]{Corollary}
\newtheorem{remark}[example]{Remark}

\newtheorem{problem}[example]{Problem}

\numberwithin{equation}{section}

\begin{document}

\title[High order isometric liftings and dilations]{High order isometric liftings and dilations}


 \author[C. Badea]{C\u{a}t\u{a}lin Badea}
 \address[C.~Badea]{Univ. Lille, CNRS, UMR 8524 - Laboratoire Paul Painlev\'{e}, France}
\email{cbadea@univ-lille.fr}

 \author[V. M\"{u}ller]{Vladimir M\"{u}ller}
  \address[V. M\"{u}ller]{Czech Academy of Sciences, Prague, Czech Republic}
  \email{muller@math.cas.cz}

\author[L. Suciu]{Laurian Suciu}
 \address[L. Suciu]{``Lucian Blaga'' University of Sibiu, Romania}
 \email{laurians2002@yahoo.com}
\thanks{The first author was supported in part by
 the project FRONT of the French
National Research Agency (grant ANR-17-CE40-0021) and by the Labex CEMPI (ANR-11-LABX-0007-01). The second author was supported by grant No. 20-31529X of GA CR and RVO:67985840 and also by Labex CEMPI during a visit to Lille. The third author was supported 
by a project financed by Lucian Blaga University of Sibiu and Hasso Plattner Foundation research Grants LBUS-IRG-2020-06. We would like to thank the referee for a careful reading of the manuscript and for useful suggestions.}
\dedicatory{
To the memory of Ciprian Foia\c {s} (1933-2020)
}

\begin{abstract}
We show that a Hilbert space bounded linear operator has an $m$-isometric lifting for some integer $m\ge 1$ if and only if the norms of its powers grow polynomially. In analogy with unitary dilations of contractions, we prove that such operators also have an invertible $m$-isometric dilation. We also study $2$-isometric liftings of convex operators and $3$-isometric liftings of Foguel-Hankel type operators.
\end{abstract}

\maketitle
\section{Introduction and preliminaries}
Denote by $\mathcal{B}(\ka)$ the algebra of all bounded linear operators acting on a Hilbert space $\ka$.
Let $\h$ be a closed subspace of a Hilbert space $\ka$. We denote by $P_{\h}\in \mathcal{B}(\ka)$ the orthogonal projection onto $\h$.
Let $T\in \mathcal{B}(\h)$ and $S\in \mathcal{B}(\ka)$. We say that $S$ is a \emph{lifting} of $T$ if $TP_{\h}=P_{\h}S$. Equivalently, in the decomposition $\ka=\h\oplus \h^\perp$, the operator $S$ has the matrix form
$$
S=\begin{pmatrix} T&0\cr *&*\end{pmatrix}.
$$
Clearly $S$ is a lifting of $T$ if and only if $S^*$ is an \emph{extension} of $T^*$, that is $S^*\h \subset \h$ and $S^*\mid \h = T^*$.

We say that $S$ is a (power) \emph{dilation} of $T$ if
$$
T^n=P_{\h}S^n \mid \h
$$
for all $n\ge 0$.

The existence of isometric liftings and unitary dilations for Hilbert space contractions are basic results in dilation theory (see for instance \cite{FF,SFBK}). For other dilation results, related to the topics studied in this paper, we refer the reader to \cite{Pet, BerPet, BadMul1, BadMul2} and the references therein.
In this paper, continuing earlier investigations in \cite{BS1,BS2,S1}, we study liftings and dilations which are $m$-isometric. Recall that $T\in \mathcal{B}(\h)$ is called $m$-\emph{isometric} for some $m\ge 1$ if it satisfies the relation
$$
\sum_{j=0}^m (-1)^j\begin{pmatrix}m\\j\end{pmatrix} T^{*j}T^j=0.
$$
Clearly $1$-isometries are just isometries in the classical meaning. We refer the reader to the trilogy \cite{AS1,AS2,AS3} for more information about $m$-isometries.

It is well known (see \cite[page 389]{AS1}) that the powers of an $m$-isometry $S$ can grow only polynomially:  there exists $K$ such that
$\|S^n\|^2\le Kn^{m-1}$ for all $n\in\N$. Therefore any operator $T$ which has an $m$-isometric lifting (or dilation) must satisfy the same estimate.

In the next section we show that any operator whose powers grow polynomially has an $m$-isometric lifting for some $m$. Moreover, it has an invertible $m'$-isometrical dilation for some odd number $m'$. 
In particular, any power bounded operator has a $3$-isometric lifting and an invertible $3$-isometric dilation.

In the last sections of the paper we indicate particular classes of operators for which one can prove stronger results. We show that all convex operators satisfying necessary growth conditions have $2$-isometric liftings, while Foguel-Hankel type operators possess $3$-isometric liftings.

\section{$m$-isometric liftings and dilations}
\subsection*{High order isometric liftings.} The next result shows that any operator whose powers grow polynomially has an $m$-isometric lifting for some $m$.
Recall that an operator $S\in \mathcal{B}(\ka)$ is called \emph{expansive} if $\|Sx\|\ge\|x\|$ for all $x\in \ka$. The operator $S$ is called \emph{analytic} if $\bigcap_{n=0}^\infty S^n(\ka)=\{0\}$.

\begin{theorem}\label{thm33}
Let $m\ge 0$ be an integer and let $T\in \B$ be an operator satisfying the condition
\begin{equation}\label{eq32}
\sup_{n\ge 1} n^{-m/2}\|T^n\|<\infty.
\end{equation}
Then $T$ has an $(m+3)$-isometric lifting.

Moreover, the $(m+3)$-isometric lifting can be chosen to be expansive and analytic.
\end{theorem}

\begin{proof}
Suppose first that the Hilbert space $\h$ is separable.

Let $K \ge \max \{1, n^{-m/2}\|T^n\| \, : n\ge 1 \}$. Then
$$
\|T^n\|^2\le K^2n^m, \quad n\ge 1.
$$
For every integer $s \ge 1$ we set
$$
\alpha_s=\left(\frac{2Ks+1}{2K(s-1)+1}\right)^{(m+2)/2}.
$$
Clearly $\alpha_1 \ge \alpha _2 \ge ... \ge 1.$

Let $\ell_+^2(\h)=\bigoplus_{j=0}^{\infty} \h_j$, where $\h_j=\h$ for $j \ge 0$, and let $S$ be the weighted forward shift of multiplicity $\dim \h$ with the weights $\alpha_s$, i.e., $S$ is defined by
$$
S(h_0,h_1,...)=(0,\alpha_1h_0,\alpha_2 h_1,...)
$$
for all sequences $(h_0,h_1,...) \in \ell_+^2(\h)$. Then
$$
\|S^n(h_0,0,...)\|^2=\|(0,0,...,(2Kn+1)^{(m+2)/2}h_0,0,...)\|^2=(2Kn+1)^{m+2}.
$$
Moreover, it is easy to see that $S$ is an $(m+3)$-isometry; see \cite{AL,AR,Richter} for more information about $m$-isometric weighted shifts.

Let $S^*$ be the adjoint of $S$, i.e., $S^*$ is the weighted backward shift defined by
$$
S^*(h_0,h_1,h_2,...)=(\alpha_1h_1,\alpha_2h_2,...).
$$
We prove now that $S^*$ is (unitarily equivalent to) an extension of $T^*$. Indeed, for $s\ge 1$, let 
$$b_s=(\alpha_1\cdot \cdot \cdot \alpha_s)^{-2}=(2Ks+1)^{-m-2}.$$ 
Using \eqref{eq32}, we get
\begin{eqnarray*}
\sum_{s=1}^{\infty} b_s \|T^{*s}\|^2&=& \sum_{s=1}^{\infty} b_s \|T^s\|^2\le K^2 \sum_{s=1}^{\infty} s^m(2Ks+1)^{-m-2}\\
&\le & K^{-m} 2^{-m-2}\sum_{s=1}^{\infty} s^{-2}\le \frac{\pi^2}{24} <1.
\end{eqnarray*}
Thus, by \cite[Theorem 2.2]{M}, $T^*$ is unitarily equivalent to a restriction of $S^*$ to an invariant subspace for $S^*$ ($\h$ being separable). Hence $S$ is an $(m+3)$-isometric lifting of $T$
 and it is clear that $S$ is analytic and expansive (because $\alpha_s\ge 1$ for all $s\ge 1$). 

If $\h$ is non-separable and $T\in \B$ satisfies the condition \eqref{eq32}, then $\h =\bigoplus_{\gamma}\h_{\gamma}$, where $\h_{\gamma}$ are separable subspaces reducing $T$. So each restriction $T|_{\h_{\gamma}}$ has an 
$(m+3)$-isometric lifting $S_{\gamma}$ on a space $\ka_{\gamma}\supset \h_{\gamma}$. Moreover, one can take for $S_{\gamma}$ the same operator on $\ka_{\gamma}=\ell_+^2(\h_{\gamma})$.
Then $S=\bigoplus_{\gamma} S_{\gamma}$ is an $(m+3)$-isometric lifting of $T$,
which is analytic and expansive.

The proof is now complete.
\end{proof}

\begin{remark}\label{re34}
\rm
In general the integer $m'=m+3$ given by the previous theorem for an operator $T$ satisfying \eqref{eq32} is not optimal, that is, sometimes it is possible for $T$ to have an $m''$-isometric lifting with $m''<m'$. Some particular cases will be discussed in the following two sections.
\end{remark}
\begin{remark}\label{re34bis}
\rm
We have the following implications:
$$
T \hbox{ has }m\hbox{-isometric lifting }\Longrightarrow \sup_n\frac{\|T^n\|^2}{n^{m-1}}<\infty \Longrightarrow
T \hbox{ has }{(m+2)}\hbox{-isometric lifting}.
$$
\end{remark}

We can therefore formulate the following corollary.

\begin{corollary}\label{co35}
For $T\in \B$ the following statements are equivalent:
\begin{itemize}
\item[(i)] $T$ has an $m$-isometric lifting for some integer $m\ge 3$;
\item[(ii)] $\{T^n\}$ satisfies a growth condition
$$
\sup_{n\ge 1}  \frac{\|T^n\|^2}{n^p}<\infty
$$
for some integer $p\ge 0$. 
\end{itemize}
\end{corollary}

The case $p=0$ in the last condition means that the operator $T$ is power bounded. Thus we obtain by Theorem \ref{thm33} the following consequence.

\begin{corollary}\label{co36}
Every power bounded operator has a 3-isometric lifting, which can be chosen to be expansive and analytic.
\end{corollary}

\subsection*{Extremal operators.} The following theorem concerns extremal operators in estimates of functions of power bounded operators. We refer to \cite{Pel} for several estimates of functions of power bounded operators on Hilbert spaces, in particular in terms of Besov-type norms, and to \cite{KM} for an answer to an open problem raised in \cite{Pel}.

In order to state the next result we need to introduce some notation. For a fixed $K>1$ and every integer $s \ge 1$ we set
$$
\alpha_s(K) =\left(\frac{2Ks+1}{2K(s-1)+1}\right).
$$
Clearly $\alpha_1(K) \ge \alpha _2(K) \ge ... \ge 1.$  Let $S_K$ be the weighted forward shift on $\ell_+^2 = \ell^2(\N,\C)$ with the weights $\alpha_s(K)$, i.e., $S_K$ is defined by
$$
S_K(z_0,z_1,...)=(0,\alpha_1(K)z_0,\alpha_2(K)z_1,...)
$$
for all sequences $(z_0,z_1,...) \in \ell_+^2$. 

\begin{theorem}\label{propo:vN}
Let $K > 1$. Suppose that $T\in \mathcal{B}(\h)$ is a power bounded operator such that $\|T^n\|\le K$ for every $n\ge 0$. Then
\begin{equation}\label{eq:vN}
\|p(T)\| \le \|p(S_K)\|
\end{equation}
for every polynomial $p$ with complex coefficients. 

However, for any fixed $K>1$, there is no power bounded weighted forward shift $E$ such that $\|f(A)\| \le \|f(E)\|$ holds true for every power bounded operator $A$ with $\sup_{n\ge 0}\|A^n\|\le K$ and every polynomial $f$. 
\end{theorem}

\begin{proof}
The first part of the proposition follows from the construction of Theorem~\ref{thm33} for $m=0$ and standard spectral theory. Indeed, the operator $S_K\otimes I_{\h}$ (the weighted forward shift of multiplicity $\dim \h$ with the weights $\alpha_s(K)$) is a dilation of $T$. 

Let $K>1$. Suppose that there exists a power bounded weighted forward shift $E$ such that $\|f(A)\| \le \|f(E)\|$ holds true for every power bounded operator $A$ with $\sup_{n\ge 0}\|A^n\|\le K$ and every polynomial $f$. It is known that a power bounded weighted forward shift is similar to a contraction (see \cite[page 55]{Shields}). Using the von Neumann inequality for contractions (see \cite[page 31]{SFBK}), we obtain that every power bounded operator $A$ with $\sup_{n\ge 0}\|A^n\|\le K$ is \emph{polynomially bounded}. This means that there is a constant $K'\ge 1$ such that $\|f(A)\| \le K'\|f\|_{\infty}$ for every polynomial $f$. That this is a contradiction can be proved as a variation of the classical construction of Foguel~\cite{Fog}. We follow the exposition from \cite{Pau}. Let $N\in \N$ be such that 
$$\frac{\frac{1}{N}+\sqrt{4+\frac{1}{N}}}{2} < K.$$
We consider the following Foguel type operator 
$$
F=
\begin{pmatrix}
S^{*} & \frac{1}{N}X\\
0 & S
\end{pmatrix}
$$
acting on $\ell_+^2\oplus \ell_+^2$, where $S$ is the forward shift of $\ell_+^2:= \ell_+^2(\C)$ and $X = \sum_{k=1}^{\infty}E_{3^k,3^k}$. Here $E_{i,j}$ denotes the standard matrix units, that is $E_{i,j}$ is $1$ in the $(i,j)$th entry and $0$ elsewhere. 
The powers of $F$ are of the following form 
$$
F^n=
\begin{pmatrix}
S^{*n} & \frac{1}{N}X_n\\
0 & S
\end{pmatrix}.
$$
It follows from the proof of \cite[Theorem 10.7]{Pau} that $\|X_n\| \le 1$. Using \cite[Corollary 2.5]{Gar} we obtain 
$$ \sup_{n\ge 0}\|F^n\| \le \frac{\frac{1}{N}+\sqrt{4+\frac{1}{N}}}{2}.$$
Thus $\sup_{n\ge 0}\|F^n\| < K$. The same proof as that of \cite[Theorem 10.9]{Pau} shows that $F$ is not polynomially bounded. We refer to \cite{Pau} for more details. Thus there is no power bounded weighted forward shift which is an extremal operator.
\end{proof}
Theorem \ref{propo:vN} should be compared to the von Neumann inequality for contractions, which says (in an equivalent form) that $\|p(C)\| \le \|p(S_+)\|$ for every contraction $C \in \mathcal{B}(\h)$. Here the extremal operator $S_+$ can be taken as the forward shift on $\ell_+^2$. 

In connection with Theorem \ref{propo:vN} and \cite{Pel}, one can ask whether the norms
$$ \|p\|_K := \sup \{\|p(T)\| : T \textrm{ is polynomially bounded, } \|T^n\| \le K (n\ge 0) \}$$
on the space of polynomials are equivalent for different values of constants $K > 1$. 

\subsection*{Invertible $m$-isometric dilations.} It is known (\cite[Proposition 1.23]{AS1}) that if $T$ is an invertible $m$-isometry and $m$ is even, then $T$ is an $(m-1)$-isometry. Suppose that $m+3$ is odd. The $(m+3)$-isometric operator $S$ constructed in Theorem \ref{thm33} has an invertible $(m+3)$-isometric extension. Indeed, assuming that 
$$
\|T^n\|^2\le K^2n^m, \quad n\ge 1,
$$
for fixed $m$ and $K$, set $w_n=(2Kn+1)^m$
for $n\in\Z$. Let $\hat S = \hat S(m,K)$ be the weighted bilateral shift of multiplicity $\dim\h$ defined by
\begin{equation}\label{eq22}
\widehat{S} (\dots, h_{-1},h_0,h_1,\dots)= \Bigl(\dots,\sqrt{\frac{w_{-1}}{w_{-2}}}h_{-2},\sqrt{\frac{w_0}{w_{-1}}}h_{-1},\sqrt{\frac{w_1}{w_{0}}}h_{0},\dots\Bigr).
\end{equation}

Clearly $\widehat{S}$ is invertible and $(m+3)$-isometric. Moreover, $\widehat{S}$ is a dilation of $T$. We obtain the following results.

\begin{theorem}
For $T\in \B$ the following statements are equivalent:
\begin{itemize}
\item[(i)] $T$ has an invertible $m$-isometric dilation for some odd integer $m\ge 3$;
\item[(ii)] $\{T^n\}$ satisfies a growth condition
$$
\sup_{n\ge 1} \frac{\|T^n\|^2}{n^p}<\infty
$$
for some integer $p\ge 0$. 
\end{itemize}
\end{theorem}

\begin{corollary} \label{C2.6}
Every power bounded operator has an invertible $3$-isometric dilation.
\end{corollary}
Since every invertible $2$-isometry is a unitary operator (see \cite[Proposition 1.23]{AS1}), Corollary~\ref{C2.6} is optimal.

We do not know the answer to the following question (which is a particular instance of what we discussed in Remark \ref{re34}).
\begin{problem}\label{problem}
Does every power bounded operator have a $2$-isometric lifting?
\end{problem}

\section{Convex operators}
In \cite{BS1}, it was proved that any concave operator $T$ (i.e., an operator satisfying $T^{*2}T^2-2T^*T+I\le 0$) has a $2$-isometric lifting. In this section we study the dual case.

We say that an operator $T\in \B$ is \emph{convex} if it satisfies the condition 
$$T^{*2}T^2-2T^*T+I\ge 0.$$
We show that convex operators satisfying the necessary growth condition \eqref{eq26} below
have $2$-isometric liftings. Note that Theorem~\ref{thm33} gives only the existence of a $4$-isometric lifting, so the result for convex operators is stronger than the general one.

We begin with the following lemma.

\begin{lemma}\label{le25}
Let $T\in \B$ be a convex operator satisfying 
\begin{equation}\label{eq26}
c:=\sup_{n\ge 0} \frac{\|T^n\|^2}{n+1}<\infty.
\end{equation}
Let $\Delta= \frac{T^{*2}T^2}{2}-T^*T+\frac{I}{2}$ and $T_1$ on $\h_1=\h \oplus \h$ be the operator given by the block matrix
$$
T_1=
\begin{pmatrix}
T & 0\\
\Delta^{1/2} & 0
\end{pmatrix}.
$$
Then the following statements hold:
\begin{itemize}
\item[(i)] $\|T_1^nh\|^2=\frac{1}{2} \bigl(\|T^{n+1}h\|^2+\|T^{n-1}h\|^2\bigr)$ for all $h\in \h \cong \h \oplus \{0\} $, $n\ge 1$;
\item[(ii)] $T_1$ satisfies \eqref{eq26}, that is $\|T_1^n\|^2\le c(n+1)$ for $n\ge 0$;
\item[(iii)] $T_1$ is convex.
\end{itemize}
\end{lemma}

\begin{proof}
$(i)$ By induction we have
$$
T_1^n =
\begin{pmatrix}
T^n & 0\\
\Delta^{1/2}T^{n-1} & 0
\end{pmatrix} 
$$
for all integers $n\ge 1$. For $h\in \h\cong \h \oplus \{0\}$ and $n\ge 1$ we obtain
\begin{eqnarray*}
\|T_1^nh\|^2&=& \|T^nh\|^2+\|\Delta^{1/2}T^{n-1}h\|^2=\|T^nh\|^2+ \bigl\langle T^{*(n-1)}\Delta T^{n-1}h,h\bigr\rangle \\
&=& \|T^nh\|^2+ \frac{\|T^{n+1}h\|^2}{2}- \|T^nh\|^2+ \frac{\|T^{n-1}h\|^2}{2}\\
 &=& \frac{1}{2} \bigl(\|T^{n+1}h\|^2+\|T^{n-1}h\|^2\bigr).
\end{eqnarray*}

$(ii)$ For $u=h\oplus h'\in \h_1$ and $n\ge 1$ we have by (i) and \eqref{eq26}, 
$$
\|T^n_1u\|^2=\|T^n_1h\|^2\le \frac{c}{2} \bigl((n+2)\|h\|^2+n\|h\|^2\bigr)=c(n+1)\|h\|^2\le c(n+1)\|u\|^2.
$$
So $\|T^n_1\|^2\le c (n+1)$.

$(iii)$ For $u=h \oplus h'\in \h_1$ we have (by (i))
$$
\|T_1^2u\|^2-2\|T_1u\|^2 +\|u\|^2 =\|T_1^2h\|^2- 2\|T_1h\|^2+\|u\|^2=
$$
$$
\|T_1^2h\|^2-(\|T^2h\|^2+\|h\|^2)+\|u\|^2\ge \|T_1^2h\|^2 -\|T^2h\|^2=\|\Delta ^{1/2}Th\|^2\ge 0.
$$
Hence $T_1$ is a convex operator on $\h_1$. This finishes the proof.
\end{proof}

In the sequel, for an operator $T$, we denote $\Delta_T=T^*T-I$.

\begin{theorem}\label{thm26}
For a convex operator $T\in \B$ the following statements are equivalent:
\begin{itemize}
\item[(i)] $T$ has a 2-isometric lifting; 
\item[(ii)] $T$ has a 2-isometric lifting $S\in \mathcal{B}(\ka)$ with $\ka \ominus \h \subset {\rm Ker}(\Delta_S)$;
\item[(iii)] $T$ satisfies the growth condition \eqref{eq26}. 
\end{itemize}
\end{theorem}

\begin{proof}
Assume that $T$ is convex. We firstly remark that the condition \eqref{eq26} is necessary for a 2-isometry (see \cite[page 389]{AS1}) and so for any operator which has a 2-isometric lifting. Therefore (i) implies (iii).

Suppose now that $T$ satisfies \eqref{eq26}. Using the previous lemma inductively we find Hilbert spaces $\h_j=\h_{j-1} \oplus \h_{j-1}$ and convex operators $T_j \in \mathcal{B}(\h_j)$ for $j\ge 1$ such that 
$$ \h \cong \h_0 \subset \h_1 \subset \h_2 \subset \cdot \cdot \cdot ,$$ $T_j$ is a lifting of $T_{j-1}$ (with $T_0=T$), $\|T_j^n\|^2\le c(n+1)$ and
\begin{equation}\label{eq27}
\|T_{j+1}^nu\|^2=\frac{1}{2}\bigl(\|T_j^{n+1}u\|^2+\|T_j^{n-1}u\|^2\bigr),
\end{equation}
for all $n\ge 1$, $j\ge 0$, $u \in \h_j$. More precisely, $\h_{j-1}$ is embedded into $\h_j$  according to the formulas 
$$\h_{j-1}\cong \h_{j-1} \oplus \{0\} \subset \h_{j-1} \oplus \h_{j-1} =\h_j.$$

Let $\ka_0=\bigcup_{j=0}^{\infty}\h_j$ and let $\ka$ be the completion of $\ka_0$. 

Let $P_j\in \mathcal{B}(\ka)$ be the orthogonal projection onto $\h_j$. 
Clearly, if $u \in \h_j$ for some $j$, then
$$
\|T_ju\|^2\le \|T_{j+1}u\|^2\le ... \le 2c\|u\|^2,
$$
so the limit $\lim_{k\to \infty} \|T_ku\|$ exists.
For $j<k$ we have
$$
\|T_{k}u-T_ju\|^2=\|T_{k}u\|^2-\|T_ju\|^2,
$$ 
so $\lim_{k\to\infty}T_ku$ exists.

For $ u\in \ka_0$ we define $Su=\lim_{j\to \infty} T_ju$. Clearly $\|S\|\le \sqrt{2c}$ and $S$ can be extended continuously to an operator acting on $\ka$ denoted by the same symbol $S$.

If $u\in\h_j$ for some $j$, then similarly
$$
\|T_j^2u\|^2\le \|T_{j+1}^2u\|^2\le \cdots \le 3c\|u\|^2,
$$
so $\lim_{j\to\infty} \|T_j^2u\|$ exists for all $u\in\ka_0$.
For $u\in\ka_0$ we have
$$
S^2u=\lim_{j\to\infty}T_jP_jSu=\lim_{j\to\infty} T_jP_jT_jP_ju=\lim_{j\to\infty} T_j^2P_ju=\lim_{j\to\infty} T_j^2u.
$$
Hence 
$$
\|S^2u\|^2-2\|Su\|^2+\|u\|^2=\lim_{j\to \infty} (\|T_j^2u\|^2-2\|T_ju\|^2+\|u\|^2)\ge 0
$$
for all $u\in\ka_0$, and so 
$S$ is a convex operator.

We show that $S$ is in fact a $2$-isometry. Suppose on the contrary that there exists $u \in \ka_0$ such that
$$
\delta:= \|S^2u\|^2-2\|Su\|^2+\|u\|^2>0.
$$
Let $j\ge 1$ be such that $u \in \h_j$ and 
$$
\|T_{j-1}^2u\|^2> \|S^2u\|-\delta.
$$
Then by \eqref{eq27} we obtain
\begin{eqnarray*}
\delta &\le & \|S^2u\|^2-2\|T_ju\|^2 + \|u\|^2\\
&=& \|S^2u\|^2-(\|T_{j-1}^2u\|^2+\|u\|^2 )+\|u\|^2\\
&=& \|S^2u\|^2-\|T_{j-1}^2u\|^2<\delta,
\end{eqnarray*}
a contradiction. Hence $\delta=0$, and so $S$ is a $2$-isometry.
Also, $S$ is a lifting for $T$ because
$$
P_{\h}Su=\lim_{j\to \infty} P_{\h}T_ju =TP_{\h}u
$$
for $u\in \ka_0$ ($T_j$ being a lifting for $T$).

In order to show that the assertion (ii) is true, we first prove that $S^*S\h\subset \h$. We have on $\h_1= \h \oplus \h$,
$$
T_1^*T_1=
\begin{pmatrix}
T^*T+\Delta & 0\\
0 & 0
\end{pmatrix}
=\frac{1}{2} 
\begin{pmatrix}
T^{*2}T^2+I_{\h} & 0\\
0 & 0
\end{pmatrix},
$$
whence it follows $T_1^*T_1\h \subset \h$ (by identifying $\h=\h\oplus \{0\}$ into $\h_1$). Also, since 
$$
T_1^{*2}T_1^2=
\begin{pmatrix}
T^{*2}T^2+T^*\Delta T& 0\\
0 & 0
\end{pmatrix},
$$
we infer $T_1^{*2}T_1^2\h \subset \h$.

On the other hand, denoting $\Delta_1= \frac{T_1^{*2}T_1^2}{2}-T_1^*T_1+\frac{I_{\h_1}}{2}$, we have that the operator $T_2$ on $\h_2=\h_1 \oplus \h_1$ has the representation 
$$
T_2=
\begin{pmatrix}
T_1 & 0\\
\Delta_1^{1/2} & 0
\end{pmatrix}.
$$
We get as above that $T_2^*T_2 = (T_1^*T_1+\Delta_1) \oplus 0=\frac{1}{2} (T_1^{*2}T_1^2+I_{\h_1})\oplus 0$. This gives (by the above inclusion) that $T_2^*T_2 \h \subset \h$, where clearly we identify $\h$ with $\h \oplus \{0\} \oplus \{0_{\h_1}\}$ in $\h_2$. By induction (and the corresponding identification of $\h$ into $\h_j$) one can see that $T_j^*T_j\h \subset \h$ for all $j\ge 1$.

Now, for every $h\in \h$ and $k \in \ka \ominus \h$, we have
\begin{eqnarray*}
\langle S^*Sh,k\rangle &=& \lim_{j\to \infty} \langle T_jh,Sk\rangle =\lim_{j \to \infty} \langle T_jh, P_jSk \rangle\\
&=& \lim_{j\to \infty} \langle T_jh, T_jP_jk \rangle =\lim_{j\to \infty} \langle T_j^*T_jh, k\rangle =0.
\end{eqnarray*}
Here we used that $S$ is a lifting for $T_j$, that $T_j^*T_jh\in \h$, $k\in \ka \ominus \h$ and $\h \subset \h_j$. Thus $S^*S\h \subset \h$.

Next, by \cite[Theorem 4.1]{BS2}, there exists a positive operator $A\in \B$ such that $T^*AT\le A$ and $\Delta_T \le A$. From these relations and the fact that $T$ is convex,\,  i.e. $\Delta_T \le T^*\Delta_TT$, we infer that
$$
\Delta_T \le T^{*n}\Delta_T T^n \le T^{*n}AT^n \le T^{*(n-1)}AT^{n-1}\le A
$$
for every integer $n\ge 1$. So it follows that the sequence $\{T^{*n}AT^n\}$ converges strongly to a positive operator $A_T\in \B$ which satisfies the relations $\Delta_T \le A_T =T^*A_TT$. Now, if $T$ is not a contraction, then $A_T \neq 0$. These relations show, by \cite[Theorem 3.2]{BS2}, that $T$ has a 2-isometric lifting $S$ on a space $\ka\supset \h$ such that $\ka \ominus \h \subset {\rm Ker}(\Delta_S)$. When $T$ is a contraction, $T$ has an isometric lifting $S$ satisfying the previous inclusion. In both cases the assertion (ii) holds. So (iii) implies (ii). As (ii) implies (i), the proof is complete. 
\end{proof}

Theorem \ref{thm26} shows that if a convex operator $T$ satisfies \eqref{eq26}, then it has a 2-isometric lifting $S$ on $\ka=\h \oplus \h^\perp$ of the form
\begin{equation}\label{eq333}
S=
\begin{pmatrix}
T & 0\\
X & V
\end{pmatrix}
\end{equation}
with $V$ an isometry on $\h^{\perp}$ satisfying $V^*X=0$ (because $\Delta_S\ge 0$). It is known from \cite{BS1,BS2} that the concave operators also have liftings of the form \eqref{eq333}, so the two dual classes of operators (convex, or concave) behave similarly in this context. Of course, every isometric lifting of a contraction is of the form \eqref{eq333} (see \cite{FF, SFBK}).

In order to obtain some applications of Theorem \ref{thm26}, we describe now the operators which have convex liftings of the form \eqref{eq333}, as well as convex liftings as those from Lemma \ref{le25}.

\begin{proposition}\label{pr33}
For an operator $T\in \B$ the following statements hold:
\begin{itemize}
\item[(a)] The operator $T$ has a convex lifting $\widehat{T}$ of the form \eqref{eq333} on $\widehat{\h}=\h\oplus \h'$ with $V=\widehat{T}|_{\h'}$ an isometry and $V^*X=0$ if and only if there exists a selfadjoint operator $A\in \B$ such that $$\Delta_T\le A\le T^*AT.$$

If this is the case and $T$ is power bounded, then $T$ has a 2-isometric lifting.

\item[(b)] The operator $T$ has a convex lifting $\widehat{T}$ on $\widehat{\h}=\h\oplus \h'$ with $\widehat{T}|_{\h'}=0$ if and only if there exists a selfadjoint operator $A\in \B$ such that $$0\le A-\Delta_T\le T^*AT-A.$$

If this is the case and $T$ satisfies the condition \eqref{eq26}, then $T$ has a 2-isometric lifting $S$ on $\ka$ such that $\ka \ominus \h\subset {\rm Ker}(\Delta_S)$.
\end{itemize}
\end{proposition} 

\begin{proof}
(a) Assume that $\widehat{T}$ is a convex lifting for $T$ as in (a). Then, using the form \eqref{eq333} for $\widehat{T}$, we get $\Delta_{\widehat{T}}=A\oplus 0$ on $\widehat{\h}=\h\oplus \h'$, where $A=X^*X+\Delta_T$ is selfadjoint. As $\widehat{T}$ is convex we have $\widehat{T}^*\Delta_{\widehat{T}}\widehat{T}\ge \Delta_{\widehat{T}}$. This, by \eqref{eq333}, implies that $T^*AT\ge A$, and clearly $A\ge \Delta_T$.

Conversely, let us assume that $\Delta_T \le A \le T^*AT$ for some selfadjoint operator $A$ on $\h$. Define the lifting $\widehat{T}$ of $T$ on the space $\widehat{\h}=\h\oplus\ell_+^2(\D)$ with the block matrix
$$
\widehat{T}=
\begin{pmatrix}
T & 0\\
(A-\Delta_T)^{1/2} & S_+
\end{pmatrix},
$$
where $\D=\overline{(A-\Delta_T)\h}$ and $S_+$ is the forward shift on $\ell_+^2(\D)$. Then $\Delta_{\widehat{T}}=A \oplus 0$ and $\widehat{T}^*\Delta_{\widehat{T}}\widehat{T}= T^*AT\oplus 0\ge A\oplus 0=\Delta_{\widehat{T}}$ (by our assumption). Hence $\widehat{T}$ is a convex lifting for $T$ of the form \eqref{eq333}. 

In addition, if $T$ is power bounded, then $\|\widehat{T}^n\|^2\le K\cdot n$, for some constant $K>0$. Indeed, we have
$$
\widehat{T}^n=
\begin{pmatrix}
T^n & 0\\
\sum_{j=0}^{n-1} V^{n-j-1}XT^j & V^n
\end{pmatrix}, \quad
n\ge 1,
$$
where $V$ is an isometry. 
So $\widehat{T}$ satisfies \eqref{eq26} and, by Theorem \ref{thm26}, $\widehat{T}$ has a 2-isometric lifting which is also a lifting of $T$.  

(b) Assume that $\widehat{T}$ is a convex lifting for $T$ on $\widehat{\h}=\h\oplus \h'$ of the form 
\begin{equation}\label{eq334}
\widehat{T}=
\begin{pmatrix}
T & 0\\
X & 0
\end{pmatrix}.
\end{equation}
Then $\Delta_{\widehat{T}}=A\oplus -I$, where $A=X^*X+\Delta_T$ is selfadjoint. Since $\widehat{T}$ is convex, we have
$$
\widehat{T}^*\Delta_{\widehat{T}}\widehat{T}= (-X^*X+T^*AT)\oplus 0\ge A\oplus-I.
$$
We infer that $T^*AT\ge A+X^*X$, or equivalently
$$
T^*AT-A\ge X^*X=A-\Delta_T\ge 0.
$$

Conversely, we suppose that there exists a selfadjoint operator $A$ in $\B$ such that $T^*AT-A\ge A-\Delta_T\ge 0$. Define $\widehat{T}$ on the space $\widehat{\h}=\h\oplus\overline{(A-\Delta_T)\h}$ as in \eqref{eq334} with $X=(A-\Delta_T)^{1/2}$. Then, as above, we have $\Delta_{\widehat{T}}=A\oplus -I$ and $\widehat{T}^*\Delta_{\widehat{T}}\widehat{T}= (T^*AT-A+\Delta_T)\oplus 0$, and our assumption ensures that $\widehat{T}^*\Delta_{\widehat{T}}\widehat{T}\ge \Delta_{\widehat{T}}$. Hence $\widehat{T}$ is a convex lifting for $T$ of the form \eqref{eq334}. In addition, if $T$ satisfies the condition \eqref{eq26}, then $\widehat{T}$ also satisfies this condition. Therefore, by Theorem \ref{thm26}, $\widehat{T}$ has a 2-isometric lifting $S$ on $\ka= \widehat{\h}\oplus \ka_0$ such that $V_0=S|_{\ka_0}$ is an isometry. Then $S$ is also a lifting for $T$, so $\h^{\perp}=\ka_0 \oplus \h'$ is an invariant subspace for $S$ and $V=S|_{\h^{\perp}}$ is a 2-isometry.

Now, taking into account the form \eqref{eq334} of $\widehat{T}$, we infer that $V$ has on $\h^{\perp}=\ka_0\oplus \h'$ the form
$$
V=
\begin{pmatrix}
V_0 & V_1\\
0 & 0
\end{pmatrix},
\quad V_0^*V_1=0.
$$
Hence $V$ is power bounded and being a 2-isometry it follows that $V$ is an isometry. This means that the lifting $S$ of $T$ has the property stated in the assertion (b).  
\end{proof}

Proposition \ref{pr33} (a) gives a partial answer to Problem \ref{problem}.

\section{Foguel-Hankel operators}
We consider now other operators which have $3$-isometric liftings but not $2$-isometric liftings. We start with a general result proving $3$-isometric liftings for some Foguel-Hankel type operators, as considered in \cite{B,DP,P, Pau}.

\begin{theorem}\label{thm37}
Let $T\in \B$ be an operator such that, with respect to an orthogonal decomposition $\h=\h_0 \oplus \h_1$,  has the block matrix form
\begin{equation}\label{eq33}
T=
\begin{pmatrix}
C_0 & C\\
0 & C_1
\end{pmatrix},
\end{equation}
where $C_i$ are contractions on $\h_i$ ($i=0,1$) and $C\in \mathcal{B}(\h_1,\h_0)$ is such that $CC_1=C_0C$. Then $T$ has a 3-isometric lifting $S$ on $\ka \supset \h$.
\end{theorem}

\begin{proof}
Consider the minimal isometric lifting $V_i$ for $C_i$ on the space $\ka_i= \h_i\oplus \ell_+^2(\D _{C_i})$, where $\D_{C_i}$ is the defect space of $C_i$, for $i=0,1$. So $V_i$ has on $\ka_i$ the block matrix form
\begin{equation}\label{eq35}
V_i=
\begin{pmatrix}
C_i & 0\\
D_i & S_i
\end{pmatrix},
\end{equation}
where $D_i =(I-C_i^*C_i)^{1/2}$ is the defect operator of $C_i$. As $CC_1=C_0C$, it follows by the commutant lifting theorem (see \cite{FF}) that there exists an operator $\widetilde{C}\in \mathcal{B}(\ka_1,\ka_0)$ such that $P_{\h_0} \widetilde{C}=CP_{\h_1}$ and $\widetilde{C}V_1=V_0\widetilde{C}$. So $\widetilde{C}$ from $\ka_1$ into $\ka_0$ has (respectively) the block matrices decompositions
\begin{equation}\label{eq36}
\widetilde{C}=
\begin{pmatrix}
C & 0\\
D & E
\end{pmatrix} : \begin{bmatrix} \h_1 \\ \h_1' \end{bmatrix} \to \begin{bmatrix} \h_0 \\ \h_0' \end{bmatrix}
\end{equation}
with some appropriate operators $D,E$.

Let $S$ on $\ka=\ka_0 \oplus \ka_1$ be the operator defined as
\begin{equation}\label{eq37}
S=
\begin{pmatrix}
V_0 & \widetilde{C}\\
0 & V_1
\end{pmatrix}.
\end{equation}
We write $S=V+Q$, where $V=V_0 \oplus V_1$ is an isometry and $Q$ a nilpotent operator of order $2$, that is satisfying $Q^2=0$. Since $\widetilde{C}V_1=V_0\widetilde{C}$, we have $VQ=QV$. Thus, by \cite[Theorem 2.2]{BMN}, $S$ is a 3-isometry. To see that $S$ is a lifting of $T$ we represent $S$ with respect to the decompositions $\ka =\h'_0 \oplus \h_0 \oplus \h_1 \oplus \h'_1$, $\ka=\h'_0\oplus \h'_1 \oplus \h_0 \oplus \h_1$ and $\ka=(\ka \ominus \h) \oplus \h$ using \eqref{eq35}, \eqref{eq36} and \eqref{eq37}. We obtain
$$
S=
\begin{pmatrix}
S_0 & D_0 & D & E\\
0 & C_0 & C & 0\\
0 & 0 & C_1 & 0\\
0 & 0 & D_1 & S_1
\end{pmatrix}
=
\begin{pmatrix}
S_0 & E & D_0 & D\\
0 & S_1 & 0 & D_1\\
0 & 0 & C_0 & C\\
0 & 0 & 0 & C_1
\end{pmatrix}
=
\begin{pmatrix}
T & 0\\
\widetilde{D} & W
\end{pmatrix}
$$
for some operators $W$ and $\widetilde D$.
\end{proof}

Recall that a {\it Foguel-Hankel operator} is an operator $T$ of the form \eqref{eq33}, where $C_1=S_+$ is a shift operator on a Hilbert space $\h$, $C_0=S_+^*$ and $C$ is a Hankel operator satisfying $CS_+=S_+^*C$. So for such operators we obtain the following consequence.

\begin{corollary}\label{co38}
Every Foguel-Hankel operator has a $3$-isometric lifting.
\end{corollary}

\begin{remark}\label{re39}
\rm 
Let us remark that one can chose $T$ of the form \eqref{eq33} with some suitable operators $C_0$, $C_1$ and $C$, such that $\|T^n\|\sim Kn$ for some constant $K>0$ and $n\ge 1$. In this case $T$ has a $3$-isometric lifting but cannot have a $2$-isometric lifting. Indeed, the existence of a $2$-isometric lifting would necessarily imply $\|T^n\|\le K_0 n^{1/2}$ for $n\ge 1$ and some $K_0>0$. 
\end{remark}

\begin{corollary}\label{co310}
Let $T,C \in \B$ be such that $T$ is a contraction and $C$ commutes with $T$. Then the operator $\widetilde{T}\in \mathcal{B}(\h\oplus \h)$ with the block matrix
\begin{equation}\label{eq38}
\widetilde{T}=
\begin{pmatrix}
T & C\\
0 & T
\end{pmatrix}
\end{equation}
has a 3-isometric lifting $S$ on some space $\ka \supset \h \oplus \h$. 
\end{corollary}

\begin{proof}
The proof follows by applying Theorem \ref{thm37} to $\widetilde{T}$ with $C_0=C_1=T$.
\end{proof}

\begin{remark}\label{re311}
\rm
In operator ergodic theory one considers $\widetilde{T}$ as in \eqref{eq38} with $C=I-T$. It follows that $\frac{\|\widetilde{T}^n\|}{n}\to 0$ ($n\to \infty$) if and only if $\|T^{n+1}-T^n\|\to 0$ ($n\to \infty$), or equivalently, the intersection of the spectrum $\sigma(T)$ of $T$ with the unit circle $\mathbb{T}$ is void, or the singleton $\{1\}$ (see \cite{TZ}).
Hence, for every contraction $T$ on $\h$ with $\sigma(T) \cap \mathbb{T} \neq \emptyset$ and $\sigma(T) \cap \mathbb{T} \neq \{1\}$, the operator $\widetilde{T}$ in \eqref{eq38} with $C=I-T$ cannot have 2-isometric liftings. Indeed, otherwise $\widetilde{T}$ would satisfy \eqref{eq32} with $m=1$ and thus $\frac{\|\widetilde{T}^n\|}{n} \to 0$, contradicting the choice of $T$. Therefore the least integer $m$ for which $\widetilde{T}$ possess an $m$-isometric lifting is $m=3$.
Notice that such operators $\widetilde{T}$ (with $C=I-T$) are not necessarily power bounded.  
\end{remark}

\begin{remark}\label{re312}
\rm
Another particular case of Corollary \ref{co310} is when $C=T$ in \eqref{eq38}. Then a 3-isometric lifting of the form \eqref{eq37} for $\widetilde{T}$ is 
$$
S=
\begin{pmatrix}
V & V\\
0 & V
\end{pmatrix}
\quad {\rm on}\quad \ka \oplus \ka,
$$
where $V$ on $\ka$ is the minimal isometric lifting of $T$. Thus, if $U$ on $\widetilde{\ka}$ is the minimal unitary extension of $V$, then the operator 
$$
\widetilde{S}=
\begin{pmatrix}
U & U\\
0 & U
\end{pmatrix}
\quad {\rm on}\quad \widetilde{\ka} \oplus \widetilde{\ka},
$$
is an extension of $S$, hence $\widetilde{S}$ is a power dilation for $\widetilde{T}$. 
Operators with $3$-isometric extensions of the form $\widetilde{S}$ (with $U$ unitary) were also studied in \cite{McCR}.
\end{remark}


\bigskip

\bibliographystyle{amsalpha}

\end{document}